\documentclass[11pt]{article}
\usepackage[hidelinks]{hyperref}
\usepackage[sort&compress]{natbib}
\usepackage{mathtools}
\usepackage{blindtext}

\usepackage{natbib}
\setcitestyle{numbers,square}

\usepackage{amssymb,amsmath,amsfonts,amsthm,array,bm,bbm,color}%
\usepackage{verbatim}
\usepackage{graphicx}
\providecommand{\U}[1]{\protect\rule{.1in}{.1in}}

\usepackage{setspace}
\newcommand{\Keywords}{\par \noindent\textbf{Keywords:}}

\numberwithin{equation}{section}
\numberwithin{equation}{section}
\newtheorem{thm}{Theorem}[section]

\newtheorem{lem}[thm]{Lemma}

\newtheorem{rmk}[thm]{Remark}
\newtheorem{define}[thm]{Definition}
\newtheorem*{lem*}{Lemma}
\newtheorem*{thm*}{Theorem}
\newcommand*{\RMN}[1]{\uppercase\expandafter{\romannumeral#1}}
\def\<{\langle}
\def\>{\rangle}
\def\d{{\rm d}}

\def\div{{\rm div \,}}
\def\curl{{\rm curl \,}}
\def\E{\mathbb{E}}

\def\P{\mathbb{P}}
\def\R{\mathbb{R}}

\title{On the pathwise uniqueness of stochastic 2D Euler equations with Kraichnan noise and $L^p$-data}
\author{Shuaijie Jiao\footnotemark[1] \quad Dejun Luo\footnotemark[2]}

\begin{document}
	\maketitle
	
	\vspace{-20pt}
	\renewcommand{\thefootnote}{\fnsymbol{footnote}}
	\footnotetext[1]{Email: jiaoshuaijie@amss.ac.cn. School of Mathematical Sciences, University of Chinese Academy of Sciences, Beijing 100049, China, and Academy of Mathematics and Systems Science, Chinese Academy of Sciences, Beijing 100190, China}
	
	\footnotetext[2]{Email: luodj@amss.ac.cn. Key Laboratory of RCSDS, Academy of Mathematics and Systems Science, Chinese Academy of Sciences, Beijing 100190, China, and School of Mathematical Sciences, University of Chinese Academy of Sciences, Beijing 100049, China}

\begin{abstract}
	In the recent work [arXiv:2308.03216], Coghi and Maurelli proved pathwise uniqueness of solutions to the vorticity form of stochastic 2D Euler equation, with Kraichnan transport noise and initial data in $L^1\cap L^p$ for $p>3/2$. The aim of this note is to remove the constraint on $p$, showing that pathwise uniqueness holds for all $L^1\cap L^p$ initial data with arbitrary $p>1$.
\end{abstract}

\section{Introduction}

We consider the vorticity form of stochastic 2D Euler equation perturbed by transport noise of Kraichnan type on $[0,T]\times \R^2$, namely
\begin{equation}\label{Euler}
	\left\{\aligned
	& \d \omega + u\cdot \nabla\omega\,\d t+ \circ\d W\cdot \nabla\omega =0,\\
	& u= \nabla^\perp (-\Delta)^{-1} \omega,
	\endaligned \right.
\end{equation}
where $\nabla^\perp=(\partial_2, -\partial_1)$, $\circ\d$ stands for the Stratonovich stochastic differential and the Kraichnan noise $W$ is a special space-time noise which is white in time, colored and divergence free in space. Specifically, the Fourier transform of its covariance matrix $Q(x-y)=Q(x,y):=\E[W(1,x) \otimes W(1,y)]$ is
\begin{equation}\label{Q^}
	\widehat{Q}(n)=\langle n\rangle^{-(2+2\alpha)}\bigg(I_2-\dfrac{n\otimes n}{\lvert n\rvert^2 }\bigg),\quad n\in \mathbb{R}^2,
\end{equation}
where $\langle n\rangle:=(1+\vert n\vert^2)^{1/2},\alpha\in (0,1)$ and $I_2$ is the $2\times 2$ identity matrix. It is worthy to mention that the noise can be represented as $W(t,x)=\sum_{k}\sigma_k(x)W_t^k$, where $\{\sigma_k\}$ is a family of divergence free vector fields and $\{W^k\}$ is a sequence of independent Brownian motion, see \cite[Section 2.1]{GalLuo23}. Since $Q(0)= \sum_k \sigma_k(x)\otimes \sigma_k(x) = \frac\pi{2\alpha} I_2$ (cf. \cite[(2.14)]{CogMau}), equation \eqref{Euler} can be formally written in the It\^o form as
\begin{equation}\label{Ito-Euler}
	\left\{\aligned
	& \d \omega + u\cdot \nabla\omega\,\d t+ \sum\limits_{k}\sigma_k\cdot\nabla\omega\,\d W^k  =\frac{\pi}{4\alpha}\Delta \omega\, \d t,\\
	& u= \nabla^\perp (-\Delta)^{-1} \omega.
	\endaligned \right.
\end{equation}

For deterministic 2D incompressible Euler equation, Yudovich established in the celebrated work \cite{Yudovich63} well-posedness of weak solutions for initial vorticity $\omega_0\in \dot{H}^{-1}\cap L^1\cap L^\infty$. Here, $\dot{H}^s= \dot{H}^s(\R^2)$ is the homogeneous Sobolev space and $L^p= L^p(\R^2)$ is the usual Lebesgue space, $p\in [1,\infty]$. If $\omega_0\in L^1\cap L^p$ with some $1<p<\infty$, DiPerna and Majda \cite{DipMaj87} have proved the existence of $L^1\cap L^p$ weak solutions, but uniqueness is still open until now. Recently, there are some ``negative'' results which imply non-uniqueness of weak solutions. Bressan and Shen showed non-uniqueness with initial vorticity in $L_{loc}^p$ by numerical method, and Vishik \cite{Vishik18-1,Vishik18-2} proved that there are infinitely many weak solutions in $L^1\cap L^p$ for 2D Euler equation in the presence of a  carefully chosen force, see also \cite{ABCDGJK} for a revisitation of the latter results.

Inspired by the regularization by noise phenomena, a natural question is whether we can find a noise, which is meaningful in physics, to restore uniqueness of $L^1\cap L^p$ solutions; see the review \cite{Flandoli15} by Flandoli for more details. Galeati and Luo \cite{GalLuo23} applied the Girsanov transform to show well-posedness in law for a general class of stochastic 2D fluid dynamical equations, including the case of finite-enstrophy solutions for stochastic logarithmically regularized 2D Euler equation, perturbed by suitable transport noise of Kraichnan type.
In the recent paper \cite{CogMau}, Coghi and Maurelli developed a new strategy to obtain weak existence of \eqref{Ito-Euler} for $\dot{H}^{-1}$ initial data and pathwise uniqueness for $\dot{H}^{-1}\cap L^1\cap L^p$ initial data with $p>3/2$; the uniqueness part is not known in the deterministic setting. The key ingredient in Coghi and Maurelli's proof is that the Kraichnan noise strongly affects the negative Sobolev norms of solutions; indeed, the noise produces an extra $H^{-\alpha}$ bound on solutions to \eqref{Ito-Euler} in the usual $\dot{H}^{-1}$ energy estimate, where $H^{-\alpha}$ is the inhomogeneous Sobolev space and $\alpha\in (0,1)$ is the index in \eqref{Q^}. This $H^{-\alpha}$ bound cancels the singularity arising from the nonlinear term and leads to pathwise uniqueness by application of the Gr\"onwall inequality.

To give a clearer statement of the results, let us recall the definition of $\dot{H}^{-1}$ weak solutions to the stochastic 2D Euler equation \eqref{Ito-Euler}.

\begin{define}
	A $\dot{H}^{-1}$ weak solution to \eqref{Ito-Euler} is an object  $(\Omega,\mathcal{A},(\mathcal{F}_t)_t,\P,(W^k)_k,\omega)$, where $(\Omega,\mathcal{A},(\mathcal{F}_t)_t,\P)$ is a filtered probability space  with the usual condition, $(W^k)_k$ is a sequence of real independent $(\mathcal{F}_t)_t$-Brownian motions, $\omega:[0,T]\times \Omega\rightarrow \dot{H}^{-1}$ is a $(\mathcal{F}_t)_t$-progressively measurable process satisfying
	\begin{equation*}
		\omega\in L_t^\infty(\dot{H}^{-1})\cap C_t(H^{-4}),\quad \P\text{-a.s.},
	\end{equation*}
	and for every $t\in [0,T]$, the following equation holds in $H^{-4}$:
	\begin{equation}\label{weak}
		\omega_t=\omega_0-\int_{0}^{t}\curl \div (u_s\otimes u_s)\, \d s-\sum_{k}\int_{0}^{t}\div (\sigma_k\omega_s)\, \d W_s^k+\frac{\pi}{4\alpha}\int_{0}^{t}\Delta\omega_s\, \d s,
	\end{equation}
	where $u=K*\omega$ and $K$ is the Biot-Savart kernel on $\R^2$.
\end{define}

\begin{rmk}
	If we know $\omega\in L_t^\infty(\dot{H}^{-1})$, then $u\in L_t^\infty(L^{2})$ and $u\otimes u\in L_t^\infty(L^{1})$. By Sobolev embedding $L^1(\R^2)\hookrightarrow H^{-1-\epsilon}(\R^2)$ for all $\epsilon>0$, we have $\curl \div (u\otimes u)\in L_t^\infty(H^{-3-\epsilon})$. Now if we neglect the stochastic integral term in \eqref{weak}, we find that $\frac{d\omega}{dt}$ is in $L_t^\infty(H^{-4})$ which implies that $\omega\in C_t(H^{-4})$. In fact, this observation still holds true when the stochastic term is considered. One can prove this rigorously using the stopping time trick as in \cite[Section 5]{CogMau}.
\end{rmk}

The main results in \cite{CogMau} are as follows:

\begin{thm*}[\cite{CogMau}, Theorems 2.11 and 2.12]
\begin{itemize}
	\item[\rm (1)]\textnormal{(Weak existence)}
	Assume the initial data $\omega_0$ is in $\dot{H}^{-1}$, then there exists a $\dot{H}^{-1}$ weak solution to \eqref{Ito-Euler},  satisfying
	\begin{equation}\label{energy-bound}
		\sup\limits_{t\in [0,T]}\E\big[\Vert \omega_t\Vert_{\dot{H}^{-1}}^2\big]
		+\int_{0}^{T}\E\big[\Vert \omega_t\Vert_{H^{-\alpha}}^2\big]\, \d t
		\lesssim \Vert \omega_0\Vert_{\dot{H}^{-1}}^2.
	\end{equation}
Moreover, if the initial datum $\omega_0$ belongs in addition to $ L^1\cap L^p$ where $1<p<\infty$, then the weak solution to \eqref{Ito-Euler} verifies also
	\begin{equation*}
		\sup_{t\in [0,T]}\big(\Vert \omega_t\Vert_{L^1}+\Vert \omega_t\Vert_{L^p}\big)
		\leq\Vert \omega_0\Vert_{L^1}+\Vert \omega_0\Vert_{L^p},\quad \P\text{-a.s.}
	\end{equation*}
	
	\item[\rm (2)]\textnormal{(Pathwise uniqueness)} Take $3/2<p<\infty$ and $\max\{0,2/p-1\}<\alpha<\min\{1-1/p,1/2\}$. Define a class
	\begin{equation*}
		\chi:= L^\infty\big(\Omega\times [0,T], L^1 \cap L^p \big)
		\cap L^\infty\big([0,T], L^2(\Omega, \dot H^{-1} \big)
		\cap L^2\big(\Omega\times [0,T], H^{-\alpha}\big).
	\end{equation*}
    Assume the initial vorticity $\omega_0$ is in $L^1\cap L^p\cap \dot{H}^{-1}$. If $\omega^1,\, \omega^2\in\chi$ are two $\dot{H}^{-1}$ weak solutions on the same probability space $(\Omega,\mathcal{A},(\mathcal{F}_t)_t,\P)$ and with respect to the same sequence of real independent $(\mathcal{F}_t)_t$-Brownian motions $(W^k)_k$, then $\omega^1=\omega^2$ $\P\text{-a.s.}$
\end{itemize}
\end{thm*}

Intuitively, as the index $\alpha$ gets smaller, the Kraichnan noise becomes spatially rougher and the regularization effect is stronger. From this point of view, the condition $\alpha>2/p-1$, resulting in the restriction $p>3/2$, is unnatural, see also \cite[Remark 2.14]{CogMau}. In this note, we can avoid this unnatural condition through a more refined estimate.
Here is our result which extends part $(2)$ in Coghi and Maurelli's theorem.

\begin{thm}[Pathwise uniqueness] \label{thm}
        Take $1<p<\infty$ and $0<\alpha<\min\{1-1/p,1/2\}$. Assume the initial data $\omega_0$ is in $L^1\cap L^p\cap \dot{H}^{-1}$. If $\omega^1,\, \omega^2\in\chi$ are two $\dot{H}^{-1}$ weak solutions to \eqref{Ito-Euler} on the same probability space $(\Omega,\mathcal{A},(\mathcal{F}_t)_t,\P)$ and with respect to the same sequence of real independent $(\mathcal{F}_t)_t$-Brownian motions $(W^k)_k$, then $\omega^1=\omega^2$ $\P\text{-a.s.}$
\end{thm}

In the next section we make some preparations needed for proving Theorem \ref{thm}, then we provide the proof in Section 3. We follow the main line of arguments in \cite[Section 7]{CogMau}, by improving the estimate on $R_2= {\rm tr}\big[(Q(0)- Q) D^2(G^\delta -G)\big] \varphi$, where $\{G^\delta \}_{\delta\in (0,1)}$ are smooth approximations of the Green function $G$ and $\varphi\in C_c^\infty(\R^2,[0,1])$ is a localizing function. The main idea is to work in the frequency space where we can show that $\widehat{R_2}$ admits a uniform bound, and vanishes as the parameter $\delta\to 0$, which further implies that $\langle R_2\ast \omega, \omega\rangle$ is infinitesimal.

\section{Preparations}

For clarity, we will adopt the same notations as in \cite{CogMau}. Sometimes, we write $a\lesssim b$ to indicate that $a\le Cb$ for some unimportant constant $C>0$.

Firstly, let us recall some fundamental properties of the Kraichnan covariance matrix $Q$ and the Green function $G$ that will be used in the sequel; for more details, see \cite{CogMau}.
	
	Take $0<\alpha<1$, the Kraichnan covariance function $Q:\mathbb{R}^2\rightarrow \mathbb{R}^{2\times2}$ is determined by its Fourier transform
	\begin{equation*}
		\widehat{Q}(n)=\langle n\rangle^{-(2+2\alpha)}\bigg(I_2-\dfrac{n\otimes n}{\lvert n\rvert^2 }\bigg),\quad n\in \mathbb{R}^2.
	\end{equation*}
The following results are proved in \cite[Proposition 2.7]{CogMau}.

	\begin{lem}\label{Q-structure}
		We have:
		\begin{equation*}
			Q(x)=B_L(\vert x \vert)\dfrac{x\otimes x}{\vert x \vert^2}+B_N(\vert x \vert)\bigg(I_2-\dfrac{x\otimes x}{\vert x \vert^2}\bigg),
		\end{equation*}
		with
		\begin{align*}
			B_L(R)&=\frac{\pi}{2\alpha}-\beta_LR^{2\alpha}-\mathrm{Rem}_{1-u^2}(R), \\
			B_N(R)&=\frac{\pi}{2\alpha}-\beta_NR^{2\alpha}-\mathrm{Rem}_{u^2}(R), \\
			\beta_N&=(1+2\alpha)\beta_L>\beta_L>0,
		\end{align*}
		where the remainders satisfy $\vert \mathrm{Rem}_{1-u^2}(R) \vert+\vert \mathrm{Rem}_{u^2}(R) \vert \lesssim R^2$ for all $R>0$. In particular, we have
		\begin{equation*}
			\vert Q(0)-Q(x) \vert \lesssim \vert x \vert^{2\alpha}\wedge1,\quad \forall x\in \mathbb{R}^2.
		\end{equation*}
	\end{lem}
	
	Let $G$ be the Green kernel of $-\Delta$ on $\R^2$, namely $-\Delta G=\delta_0$; one has
	\begin{equation*}
		G(x)=-\frac{1}{2\pi}\log \vert x\vert ,\quad x\neq 0.
	\end{equation*}
    The Biot-Savart kernel is given by
      \[K(x):=\nabla^\perp       G(x)=-\frac{1}{2\pi}\dfrac{x^\perp}{\vert x \vert^2},\quad x\neq 0.\]
    It is easy to check the Hessian matrix of $G$ is
	\begin{equation*}
		D^2G(x)=-\frac{1}{2\pi}\frac{1}{\vert x\vert^2}\bigg(I_2-\dfrac{2x\otimes x}{\vert x \vert^2}\bigg),\quad x\neq0.
	\end{equation*}
	Now we give the definition of the approximation kernel $G^{\delta}$ with $0<\delta<1$. Let $p(t,x)=(4\pi t)^{-1}e^{-\vert x\vert^2/4t}$ be the heat kernel on $\R^2$, it is well known that $G$ and $p$ is related by
	\begin{equation*}
		G(x)=\int_{0}^{\infty}p(t,x)\, \d t.
	\end{equation*}
    Define the smooth approximation kernel $G^{\delta}$ for $0<\delta<1$ as
    \begin{equation*}
    	G^{\delta}(x)=\int_{\delta}^{1/\delta}p(t,x)\, \d t.
    \end{equation*}
	It is easy to check $\widehat{G^\delta}(n)=(2\pi\vert n\vert)^{-2}(e^{-4\pi^2\vert n\vert^2\delta}-e^{-4\pi^2\vert n\vert^2/\delta})$, which implies that $G^{\delta}$ is a Schwartz function. From \cite[Lemma 3.4]{CogMau}, the Hessian matrix is
	\begin{equation*}
		D^2G^{\delta}(x)
		=\frac{1}{2\pi}\frac{1}{\vert x\vert^2}
		\bigg[\bigg(-I_2+\dfrac{2x\otimes x}{\vert x \vert^2}+\dfrac{\delta x\otimes x}{2}\bigg)
		e^{-\frac{\vert x\vert^2\delta}{4}}
		-\bigg(-I_2+\dfrac{2x\otimes x}{\vert x \vert^2}+\dfrac{x\otimes x}{2\delta }\bigg)e^{-\frac{\vert x\vert^2}{4\delta}}\bigg].
	\end{equation*}

	Now we give three simple lemmas that will be used to control the term $R_2$ in the proof of Theorem \ref{thm}. First, let $M$ denote the Hardy-Littlewood maximal operator, that is to say, for $f\in L_{loc}^1(\R^2)$, define
	\begin{equation*}
		Mf(x):=\sup\limits_{r>0} M_r f(x):=\sup\limits_{r>0}\frac{1}{r^2}\int_{B_r(x)}\vert f(y)\vert\, \d y,
	\end{equation*}
    where $B_r(x)=\{y\in \R^2:\vert y-x\vert<r\}$.

	\begin{lem}\label{maximal}
Let $0<\alpha<1$ and $g(x)=\langle x\rangle^{-2\alpha}$, then there exists a constant $C=C_{\alpha}$ such that $Mg(x)\leq Cg(x)$ for all $x\in\R^2$.
	\end{lem}
	
	\begin{proof}
		 It is trivial that $Mg$ is bounded since $g$ is bounded, so it suffices to prove the estimate for all $x\in \R^2$ such that $\vert x\vert>1$.
		 Arbitrarily fix a point $x\in \R^2$ with $\vert x\vert>1$. On the one hand, if $r\leq\vert x\vert/2$, then for all $y\in B_r(x)$, we have $\vert y\vert \geq \vert x\vert/2$, thus
		 \begin{equation*}
		 	M_rg(x)=\frac{1}{r^2}\int_{B_r(x)}\langle y\rangle^{-2\alpha}\, \d y
		 	\leq C\langle x\rangle^{-2\alpha}.
		 \end{equation*}
	     On the other hand, if $r>\vert x\vert/2>1/2$, then $B_r(x)\subset B_{3r}(0)$, thus
	     \begin{equation*}
	     	\begin{split}
	     		M_rg(x)&\leq \frac{1}{r^2}\int_{B_{3r}(0)}\langle y\rangle^{-2\alpha}\, \d y
	     		=\frac{2\pi}{r^2}\int_{0}^{3r}\frac{s}{(1+s^2)^\alpha}\, \d s  \\
	 	     	&= \frac\pi{(1-\alpha)r^2} \big[(1+9r^2)^{1-\alpha}-1 \big]
	     		\leq C_{\alpha}\langle r\rangle^{-2\alpha}
	     		\leq C_{\alpha}\langle x\rangle^{-2\alpha}.
	     	\end{split}
	     \end{equation*}
	     Combining the above two estimates, we get the desired result.
	\end{proof}

      The following results are taken from \cite[Proposition 1.16, Remark 1.17]{Bahouri2011FourierAA}.
    \begin{lem}\label{dominate1}
    	Suppose $\psi\in L^1(\R^2)$ is nonnegative, radial and nonincreasing, then for all measurable functions $f$, we have
    	\begin{equation*}\label{special}
    		\vert \psi*f(x)\vert \leq \Vert \psi\Vert_{L^1}Mf(x),\quad \forall x\in \R^2.
    	\end{equation*}
        Moreover, for a general measurable function $K$, we have
        \begin{equation*}
        	\vert \psi*f(x)\vert \leq \Vert S(\psi)\Vert_{L^1}Mf(x),\quad \forall x\in \R^2,
        \end{equation*}
        where $S(\psi)(y):=\sup\limits_{\vert z\vert>\vert y\vert}\vert \psi(z)\vert$ is nonnegative, radial and nonincreasing.
    \end{lem}

    \begin{rmk}\label{remark}
    	For all $t>0$, define $\psi_t(x)=t^{-2} \psi(t^{-1}x)$, then we have $S(\psi_t)=(S(\psi))_t$. Indeed, for all $y\in \R^2$, we have
    	\begin{equation*}
    		S(\psi_t)(y)
    		=\sup\limits_{\vert z\vert>\vert y\vert}\vert \psi_t(z)\vert
    		=t^{-2}\sup\limits_{t^{-1}\vert z\vert>t^{-1}\vert y\vert}
    		\vert \psi(t^{-1}z)\vert
    		=t^{-2}S(\psi)(t^{-1}y)=(S(\psi))_t(y).
    	\end{equation*}
    \end{rmk}

    \begin{lem}\label{converge}
    	Suppose $\rho\in L^1(\R^2)$ and $\int \rho=a\in \R$. Set $\rho_t(x)=t^{-2}\rho(t^{-1}x)$, then we have
    	\begin{itemize}
    		\item[\rm (1)] if $f$ is bounded and uniformly continuous, then $\rho_t*f\rightarrow af$ uniformly as $t\rightarrow 0$;
    		\item[\rm (2)] if $f$ is continuous and vanishes at infinity, then $\rho_t*f\rightarrow 0$ pointwise as $t\rightarrow \infty$.
    	\end{itemize}
    \end{lem}

    \begin{proof}
    	$(1)$ is well known by approximation of identity. Regarding $(2)$, for any $x\in \R^2$ fixed, by change of variables, we have
    	\begin{equation*}
    		\rho_t*f(x)=\int f(x-y)\rho_t(y) \, \d y
    		=\int f(x-ty)\rho(y)\, \d y,
    	\end{equation*}
        which converges to $0$ as $t\rightarrow \infty$ by dominated convergence theorem.
    \end{proof}

\section{Proof of the main result}

\begin{proof}
    	Assume that $\omega^1,\, \omega^2\in \chi$ are two $\dot{H}^{-1}$ weak solutions to \eqref{Ito-Euler} on the same probability space $(\Omega,\mathcal{A},(\mathcal{F}_t)_t,\P)$ and with respect to the same sequence of real independent $(\mathcal{F}_t)_t$-Brownian motions $(W^k)_k$. Then the difference $\omega:=\omega^1-\omega^2$ satisfies the following equality in $H^{-4}$:
    	\begin{equation*}
    		\d \omega
    		+\big[(K*\omega^1)\cdot\nabla\omega+(K*\omega)\cdot \nabla\omega^2\big]\,\d t
    		+\sum_{k}\sigma_k\cdot\nabla\omega\, \d W^k
    		=\frac{\pi}{4\alpha}\Delta\omega\,\d t.
    	\end{equation*}
    	Recall the smooth kernel $G^\delta$ defined in Section 2, applying It\^o's formula to $\langle \omega,G^\delta*\omega\rangle$ and integrating by parts yield
    	\begin{equation}\label{Ito-formula}
    		\begin{split}
    			\d \langle \omega,G^\delta*\omega\rangle
    			&=2\langle \nabla G^\delta*\omega, (K*\omega^1)\omega\rangle\,\d t
    			+2\langle\nabla G^\delta*\omega, (K*\omega)\omega^2\rangle\,\d t  \\
    			&\quad+2\sum_{k}\langle \nabla G^\delta*\omega,\sigma_k\omega\rangle\,\d W^k
    			+\frac{\pi}{2\alpha}\langle G^\delta*\omega,\Delta\omega\rangle \,\d t\\
    			&\quad+\sum_k\langle \sigma_k\cdot\nabla\omega, G^\delta*(\sigma_k\cdot\nabla\omega)\rangle\,\d t\\
    			&=:2I_1\,\d t+2I_2\,\d t+\d M+J\,\d t.
    		\end{split}
    	\end{equation}
        Since the kernel $G^\delta$ is a Schwartz function, one easily checks that the stochastic integral $M$ is a true martingale and vanishes once taking expectation.

	        Concerning the nonlinear term $I_1$, under the condition $\alpha<\min\{1/2,1-1/p\}$, Coghi and Maurelli showed in \cite[Section 7]{CogMau} that, for every $0<\epsilon\ll1$, there exists a positive constant $C_\epsilon=C(\epsilon,\alpha,\Vert \omega_0\Vert_{L^1}+\Vert \omega_0\Vert_{L^p})$ such that
        \begin{equation}\label{I_1}
        	\int_{0}^{t}\E[\vert I_1\vert]\,\d s
        	\leq \epsilon\int_{0}^{t}\E\big[\Vert \omega_s\Vert_{H^{-\alpha}}^2\big]\, \d s
        	+ C_\epsilon\int_{0}^{t}\E\big[\Vert \omega_s\Vert_{\dot{H}^{-1}}^2\big]\, \d s.
        \end{equation}

        For the nonlinear term $I_2$, noting that $\langle\nabla G*\omega, (K*\omega)\omega^2\rangle=\langle\nabla (G*\omega), \nabla^\perp (G*\omega)\omega^2\rangle=0$, we have $I_2=\langle\nabla (G^\delta-G)*\omega, (K*\omega)\omega^2\rangle$.
        So intuitively, $I_2$ should converge to zero as $\delta\rightarrow 0$. Indeed, by results in \cite[Section 7]{CogMau}, we have
        \begin{equation}\label{I_2}
        	\int_{0}^{T}\E[\vert I_2\vert]\,\d s\rightarrow0,\quad\text{as }\delta\rightarrow0.
        \end{equation}

        Our task now is to compute the term
        \begin{equation*}
        	J=\frac{\pi}{2\alpha}\langle G^\delta*\omega,\Delta\omega\rangle
        	+\sum_k\langle \sigma_k\cdot\nabla\omega, G*(\sigma_k\cdot\nabla\omega)\rangle,
        \end{equation*}
        which will cancel out the $L^2\big([0,t]\times \Omega,H^{-\alpha}\big)$-norm in the estimate of $I_1$. Using integration by parts and recalling that $Q(0)=\frac{\pi}{2\alpha}I_2$, we get
        \begin{equation*}
        	\frac{\pi}{2\alpha}\langle G^\delta*\omega,\Delta\omega\rangle=\frac{\pi}{2\alpha}\langle \Delta G^\delta*\omega,\omega\rangle
        	=\big\langle\mathrm{tr}\big[Q(0)D^2G^\delta\big]*\omega,\omega\big\rangle.
        \end{equation*}
        Integrating by parts again leads to
        \begin{equation*}
        	\begin{split}
        		\sum_k\langle \sigma_k\cdot\nabla\omega, G^\delta*(\sigma_k\cdot\nabla\omega)\rangle
        		&=\sum_{k}\big\langle \nabla\cdot(\sigma_k\omega),\nabla\cdot \big(G^\delta*(\sigma_k\omega)\big)\big\rangle \\
        		&=-\sum_{k}\langle \sigma_k\omega,D^2 G^\delta*(\sigma_k\omega)\rangle \\
        		&=-\sum_{i,j}\iint\partial_{ij}^2G^\delta(x-y)\sum_{k}\sigma_k^i(x)\sigma_k^j(y) \omega(x)\omega(y)\,\d x\d y \\
        		&=-\big\langle\mathrm{tr}\big[QD^2G^\delta\big]*\omega,\omega\big\rangle.
        	\end{split}
        \end{equation*}
        Hence we have
        \begin{equation*}
        	J=\big\langle\mathrm{tr}\big[(Q(0)-Q)D^2G^\delta\big]*\omega,\omega\big\rangle.
        \end{equation*}
    	Let $\varphi$ be a radial smooth function satisfying $0\leq \varphi \leq 1$ everywhere, $\varphi(x)=1$ for $\vert x\vert\leq1$ and $\varphi(x)=0$ for $\vert x\vert\geq2$, we split the term $\mathrm{tr}\big[(Q(0)-Q)D^2G^\delta\big]$ as follows:
    	\begin{equation*}
    		\begin{split}
    			\mathrm{tr}\big[(Q(0)-Q)D^2G^\delta\big]
    			&=\mathrm{tr}\big[(Q(0)-Q)D^2G\big]\varphi  \\
    			&\quad+\mathrm{tr}\big[(Q(0)-Q)D^2(G^\delta-G)\big]\varphi \\
    			&\quad+\mathrm{tr}\big[(Q(0)-Q)D^2G^\delta\big](1-\varphi)  \\
    			&=:A+R_2+R_3.
    		\end{split}
    	\end{equation*}
    	By \cite[Lemmas 4.3 and 4.5]{CogMau}, there exists two positive constants $c$ and $C$ such that for all $n\in \R^2$,
    	\begin{equation}\label{A,R_3}
    		\widehat{A}(n)\leq -c\langle n\rangle^{-2\alpha}+C\langle n\rangle^{-2},\quad
    		\widehat{R_3}(n)\leq C\vert n\vert^{-2}.
    	\end{equation}
    	
    	It remains to estimate $R_2$.
    	
    \begin{lem}\label{lem-3}
    We have $\vert\widehat{R_2}(n)\vert\leq C\langle n\rangle^{-2\alpha}$, where the constant $C$ is independent of $\delta$ and $n$; moreover, $\vert\widehat{R_2}(n)\vert\rightarrow 0\text{ as }\delta\rightarrow0$ for all $n\in \R^2$.
    \end{lem}
    	
    	Assume the assertions hold for the moment and we continue the proof of Theorem \ref{thm}. By Lemma \ref{lem-3}, we have
    	\begin{equation*}
    		\vert\langle R_2*\omega_t,\omega_t\rangle\vert
    		\leq \int \vert\widehat{R_2}(n)\vert\,\vert \hat{\omega}_t(n)\vert^2\,\d n
    		\leq C\Vert \omega_t\Vert_{H^{-\alpha}}^2,
    	\end{equation*}
        which is integrable in $[0,T]\times \Omega$ since the $L^2\big(\Omega\times [0,T], H^{-\alpha}\big)$ norm of $\omega$ is finite. Moreover, since $\vert\widehat{R_2}(n)\vert\rightarrow 0\text{ as }\delta\rightarrow0$, the dominated convergence theorem implies
    	\begin{equation}\label{R_2}
    		\int_{0}^{T}\E\big[\vert\langle R_2*\omega_t,\omega_t\rangle\vert\big]\,\d t
    		\leq\int_{0}^{T}\E\bigg[\int \vert\widehat{R_2}(n)\vert\,\vert \hat{\omega}_t(n)\vert^2\,\d n\bigg]\,\d t \rightarrow0,\quad\text{as }\delta\rightarrow0.
    	\end{equation}
        The estimates \eqref{A,R_3} and \eqref{R_2} yield that
        \begin{equation}\label{J}
        	\begin{split}
        		\int_{0}^{t}\E[J]\,\d s
        		&=\int_{0}^{t}\E\bigg[\int \big(\widehat{A}+\widehat{R_3}+\widehat{R_2}\big)(n)\,\vert \hat{\omega}_s(n)\vert^2\,\d n\bigg]\,\d s \\
        		&\leq \int_{0}^{t}\E\bigg[\int \big(-c\langle n\rangle^{-2\alpha}+C\vert n\vert^{-2}\big)\,\vert \hat{\omega}_s(n)\vert^2\,\d n\bigg]\,\d s+o(1) \\
        		&= -c\int_{0}^{t}\E\big[\Vert \omega_s\Vert_{H^{-\alpha}}^2\big]\, \d s
        		+ C\int_{0}^{t}\E\big[\Vert \omega_s\Vert_{\dot{H}^{-1}}^2\big]\, \d s+o(1).
        	\end{split}
        \end{equation}
        Integrating over $[0,t]\times \Omega$ in \eqref{Ito-formula} and letting $\delta\rightarrow0$, using the bounds \eqref{I_1} with $\epsilon=c/2$, \eqref{I_2} and \eqref{J}, we obtain
        \begin{equation*}
        	\mathbb{E}\big[\Vert\omega_t\Vert_{\dot{H}^{-1}}^2\big]\leq \Vert\omega_0\Vert_{\dot{H}^{-1}}^2
        	-\frac{c}{2}\int_{0}^{t}\mathbb{E}\left[\Vert\omega_s\Vert_{H^{-\alpha}}^2\right]\mathrm{d}s
        	+C \int_{0}^{t}\mathbb{E}\big[\Vert\omega_s\Vert_{\dot{H}^{-1}}^2\big]\mathrm{d}s.
        \end{equation*}
    	Since $\omega_0=\omega_0^1-\omega_0^2=0$, by Gr\"onwall inequality, we get
    	\begin{equation*}
    		\sup\limits_{t\in [0,T]}\E\big[\Vert\omega_t\Vert_{\dot{H}^{-1}}\big]
    		=0,
    	\end{equation*}
    	which implies $\omega^1=\omega^2$ and completes the proof of Theorem \ref{thm}.
    	
    	Now we turn to proving Lemma \ref{lem-3}. Recall the expressions of $D^2G$ and $D^2G^\delta$; we have
    	\begin{equation*}
    		(D^2G^\delta-D^2G)(x)
    		=D^2G(x)\big(e^{-\frac{\delta\vert x\vert^2}{4}}-e^{-\frac{\vert x\vert^2}{4\delta}}-1\big)
    		+ \frac{x\otimes x}{4\pi\vert x\vert^2} \big(\delta e^{-\frac{\delta\vert x\vert^2}{4}}- \delta^{-1} e^{-\frac{\vert x\vert^2}{4\delta}} \big).
    	\end{equation*}
        Substituting it into $R_2$, by the structure of the covariance $Q$ in Lemma \ref{Q-structure}, we have
        \begin{equation*}
        	\begin{split}
        		R_2(x)&=\mathrm{tr}\big[(Q(0)-Q(x))D^2(G^\delta-G)(x)\big]\varphi(x) \\
        		&=\mathrm{tr}\big[(Q(0)-Q(x))D^2G(x)\big]\varphi(x)\big(e^{-\frac{\delta\vert x\vert^2}{4}}-e^{-\frac{\vert x\vert^2}{4\delta}}-1\big) \\
        		&\quad+	\frac{1}{4\pi}\mathrm{tr}\Big[\Big(B_L(0)I_2-B_L(\vert x\vert)\frac{x\otimes x}{\vert x\vert^2}-B_N(\vert x\vert)\Big(I_2-\frac{x\otimes x}{\vert x\vert^2}\Big)\Big)\frac{x\otimes x}{\vert x\vert^2}\,\Big] \\
          &\qquad \times \varphi(x) \big(\delta e^{-\frac{\delta\vert x\vert^2}{4}}-\delta^{-1}e^{-\frac{\vert x\vert^2}{4\delta}}\big) \\
        		&=A(x)\big(e^{-\frac{\delta\vert x\vert^2}{4}}-e^{-\frac{\vert x\vert^2}{4\delta}}-1\big)
        		+\frac{1}{4\pi}\big(B_L(0)-B_L(\vert x\vert)\big)\varphi(x)\big(\delta e^{-\frac{\delta\vert x\vert^2}{4}}-\delta^{-1}e^{-\frac{\vert x\vert^2}{4\delta}}\big) \\
        		&=:R_{21}(x)+R_{22}(x).
        	\end{split}
        \end{equation*}

        For the term $R_{21}$, setting $h(x)=e^{-\frac{\vert x \vert^2}{4}}$, we have $R_{21}(x)=A(x)\big(h(\delta^{1/2}x)-1-h(\delta^{-1/2}x)\big)$, therefore
        \begin{equation*}
        	\begin{split}
        		\vert \widehat{R_{21}}(n)\vert
        		&=\big\vert \big(\widehat{A}*\hat{h}_{\delta^{1/2}}-\widehat{A}-\widehat{A}*\hat{h}_{\delta^{-1/2}}\big)(n) \big\vert\\
        		&\leq \big\vert \big(\widehat{A}*\hat{h}_{\delta^{1/2}}-\widehat{A}\,\big)(n) \big\vert
        		+\big\vert \widehat{A}*\hat{h}_{\delta^{-1/2}}(n)\big\vert.
        	\end{split}
        \end{equation*}
        where $\hat{h}_{\delta^{1/2}}(\cdot)=\delta^{-1}\hat{h}(\delta^{-1/2}\cdot)$ and $\hat{h}_{\delta^{-1/2}}(\cdot)=\delta\hat{h}(\delta^{1/2}\cdot)$, in accordance with the notation in Lemma \ref{converge}. Note that $\int \hat{h}(n)\, \d n=h(0)=1$ and $\hat{h}(n)=4\pi e^{-4\pi^2\vert n\vert^2}$ for all $n\in \R^2$. On the one hand, observe that the support of $A$ is compact, and $\vert A(x)\vert\lesssim \vert x\vert^{2\alpha-2}$ near the origin, hence $A$ is integrable on $\R^2$; this implies that $\widehat{A}$ is continuous and, by the well known Riemann-Lebesgue lemma, it vanishes at infinity. As a result, we are now in the position to apply Lemma \ref{converge}, where $\widehat{A}$ and $\hat{h}$ correspond to $f$ and $\rho$, respectively. Hence we obtain that for every $n\in\R^2$, as $\delta\rightarrow 0$,
        \begin{equation*}
        	\widehat{A}*\hat{h}_{\delta^{1/2}}(n)-\widehat{A}(n)\rightarrow 0 \quad\text{and}\quad \widehat{A}*\hat{h}_{\delta^{-1/2}}(n)\rightarrow 0.
        \end{equation*}
        On the other hand, following the proof of the first estimate in \eqref{A,R_3} (see \cite[Lemma 4.3]{CogMau}), one can also show that $|\widehat{A}(n)| \lesssim \langle n\rangle^{-2\alpha}$; Lemma \ref{maximal} implies $M\widehat{A}(n) \lesssim \langle n\rangle^{-2\alpha}$. Moreover, note that
          $$\big\| \hat{h}_{\delta^{1/2}} \big\|_{L^1} = \big\| \hat{h}_{\delta^{-1/2}} \big\|_{L^1} = \|\hat h\|_{L^1} =1$$
        is independent of $\delta \in (0,1)$, thus applying Lemma \ref{dominate1} with $\psi=\hat{h}_{\delta^{1/2}}$ and $\psi=\hat{h}_{\delta^{-1/2}}$, respectively, we deduce that there exists a constant $C>0$ such that for every $0<\delta<1\text{ and } n\in\R^2$,
        \begin{equation*}
        	\big\vert\widehat{A}*\hat{h}_{\delta^{1/2}} (n)\big\vert \leq C\langle n \rangle^{-2\alpha}\quad \text{and}\quad \big\vert\widehat{A}*\hat{h}_{\delta^{-1/2}} (n)\big\vert \leq C\langle n \rangle^{-2\alpha}.
        \end{equation*}
        Hence we have $\vert\widehat{R_{21}}(n)\vert\lesssim\langle n\rangle^{-2\alpha}$ and $\vert\widehat{R_{21}}(n)\vert\rightarrow 0 \text{ as }\delta\rightarrow0$ for all $n\in \R^2$.

        Now let us turn to coping with the term $R_{22}$, which can be split as follows:
        \begin{equation*}
        	\begin{split}
        		R_{22}
        		&=\frac{1}{4\pi\vert x \vert^2}\big(B_L(0)-B_L(\vert x\vert)\big)\varphi(x)\,
        		\big(\delta\vert x\vert^2e^{-\frac{\delta\vert x\vert^2}{4}}- \delta^{-1}\vert x\vert^2e^{-\frac{\vert x\vert^2}{4\delta}}\big)   \\
        		&=:\bar{A}(x)\, (H(\delta^{1/2}x)-H(\delta^{-1/2}x)),
        	\end{split}
        \end{equation*}
        where $H(x)=\vert x\vert^2e^{-\frac{\vert x\vert^2}{4}}$. Then we have
        \begin{equation*}
        		\begin{split}
        		\vert \widehat{R_{22}}(n)\vert
        		&=\big\vert \big(\mathcal{F}(\bar{A})*\widehat{H}_{\delta^{1/2}}-\mathcal{F}(\bar{A})*\widehat{H}_{\delta^{-1/2}}\big)(n) \big\vert\\
        		&\leq \big\vert \big(\mathcal{F}(\bar{A})*\widehat{H}_{\delta^{1/2}}\big)(n) \big\vert
        		+\big\vert \mathcal{F}(\bar{A})*\widehat{H}_{\delta^{-1/2}}(n)\big\vert.
        	\end{split}
        \end{equation*}
        Note that $\widehat{H}$ is a Schwartz function and $\int\widehat{H}(n)\,\d n=H(0)=0$. Remark \ref{remark} implies that the $L^1$-norm of $S\big(\widehat{H}_{\delta^{1/2}} \big)$ and $S\big(\widehat{H}_{\delta^{-1/2}} \big)$  is a finite constant independent of $\delta$.
        Next, by Lemma \ref{Q-structure}, $\bar{A}$ enjoys the same properties as $A$, and thus its Fourier transform $\mathcal{F}(\bar{A})$ is continuous on $\R^2$ and vanishes at infinity; moreover, repeating the proof of \cite[Lemma 4.3]{CogMau}, it is easy to verify that $\mathcal{F}(\bar{A})$ has similar pointwise upper bound, namely $\vert \mathcal{F}({\bar{A}})(n)\vert \lesssim \langle n\rangle ^{-2\alpha}$ for all $n\in\mathbb{R}^2$.
        Thus we can handle this term in exactly the same way as term $R_{21}$, with $\bar{A}$ and $H$ corresponding to $A$ and $h$, respectively, so $\vert\widehat{R_{22}}(n)\vert\lesssim\langle n\rangle^{-2\alpha}$ and $\vert\widehat{R_{22}}(n)\vert\rightarrow 0 \text{ as }\delta\rightarrow0$ for all $n\in \R^2$. Combining the estimates of $R_{21}$ and $R_{22}$, Lemma \ref{lem-3} is thus proved and we obtain the desired assertion.
    \end{proof}

    \begin{rmk}
    	The key point is that we just use the $H^{-\alpha}$ bound of $\omega$ to control the remainder term $R_2$, rather than exploiting the $L^p$ bound and Young inequality as in \cite[Section 7]{CogMau}. The latter is quite straightforward, but requires the embedding $L^p\hookrightarrow H^{-\alpha}$, which leads to the restriction of the range of $p$. In fact,
    	the condition $\alpha>2/p-1$ is equivalent to $p>2/(1+\alpha)$, hence due to \cite[Remark 2.13]{CogMau}, the class $L^2\big([0,T]\times \Omega;H^{-\alpha}(\R^2)\big)$ is already redundant in the statement of uniqueness part of Theorem 2.12 there.
    \end{rmk}

    \begin{rmk}
    	In the recent work \cite{JiaoLuo24}, we considered stochastic mSQG (modified Surface Quasi-Geostrophic) equations with Kraichnan noise, where the constraint similar to $\alpha>2/p-1$ still exists, see \cite[Theorem 1.4]{JiaoLuo24}. The tricks in this note do not work in that case, since the approximate kernel $G_\beta^\delta$ therein is constructed by fractional heat kernel, for which the corresponding approximate kernel in physical space does not have an explicit formula. We also mention that Bagnara et al. \cite{BGM24} studied the same equation but they assumed the initial data belong to $L^1\cap L^p$ with $p\ge 2$.
    \end{rmk}

\medskip
		
\noindent \textbf{Acknowledgements.}
The second author is grateful to the financial supports of the National Key R\&D Program of China (No. 2020YFA0712700), the National Natural Science Foundation of China (Nos. 11931004, 12090010, 12090014), and the Youth Innovation Promotion Association, CAS (Y2021002).

\end{document}